\documentclass[a4paper,12pt]{article}
\usepackage[utf8]{inputenc}
\usepackage[T1]{fontenc}
\usepackage{amsmath,amsfonts,amssymb,amsthm,amscd,tipa,color,graphicx,tikz-cd,longtable}
\usepackage[english]{babel}
\usepackage[a4paper,margin=0.75in,top=4.5cm,bottom=3.5cm]{geometry}
\usepackage[all]{xy}
\usepackage{multirow}

\newcommand{\Z}{\mathbb{Z}}
\newcommand{\Q}{\mathbb{Q}}
\newcommand{\R}{\mathbb{R}}

\newcommand{\C}{\mathbb{C}}

 \DeclareMathOperator{\psl}{PSL}
 \DeclareMathOperator{\slc}{SL}
 \DeclareMathOperator{\pgl}{PGL}
 \DeclareMathOperator{\gl}{GL}
 
 \DeclareMathOperator{\PU}{PU}
  \DeclareMathOperator{\SU}{SU}
   \DeclareMathOperator{\SO}{SO}
   \DeclareMathOperator{\PO}{PO}
 \DeclareMathOperator{\Gal}{Gal}

\theoremstyle{plain}
\begingroup

\newtheorem{theo}{Theorem}[section]
\newtheorem*{theo*}{Theorem}

\newtheorem{lemma}[theo]{Lemma}

\newtheorem{rk}[theo]{Remark}
\newtheorem{prop}[theo]{Proposition}

\title{}
\author{}
\date{\vspace{-3ex}}

\begin{document}
\title{Representations of Deligne-Mostow lattices into $\pgl(3,\C)$}

\author{E. Falbel, I. Pasquinelli and A. Ucan-Puc}

\maketitle

\abstract{We classify representations of a class of Deligne-Mostow lattices into $\pgl(3,\C)$. In particular, we show local rigidity for the representations (of Deligne-Mostow lattices with 3-fold symmetry and of type one) where the generators we chose are of the same type as the generators of Deligne-Mostow lattices. 
We also show local rigidity without constraints on the type of generators for six of them and we show the existence of local deformations for a number of representations in three of them.
We use formal computations in 
 SAGE and Maple to obtain the results.
The code files are available on GitHub (\cite{GH}). 
\section{Introduction}

Lattices in semi-simple Lie groups have strong rigidity properties (see \cite{WM} for a general introduction to lattices in semi-simple Lie groups and more details in Section \ref{sec:rigidity}). The Mostow-Prasad rigidity theorem states that, when $\rho : \Gamma_1\rightarrow \Gamma_2$ is an isomorphism between two lattices $\Gamma_1$, $\Gamma_2$ contained in  $G_1$ and $G_2$ (with $G_i$ being a connected simple Lie group with trivial center and non-isogenous to $\slc(2, \R)$), then $\rho$ extends to an isomorphism between the Lie groups $G_1$ and $G_2$.
%in dimension greater than two, any cocompact (or finite covolume by Prasad) lattice $\Gamma\subset G$ is rigid in the sense that an isomorphism between lattices extends to a homomorphism between isometry groups.
In higher rank, Margulis' superrigidity theorem implies that any homomorphism $\phi: \Gamma\rightarrow G'$ of a lattice $\Gamma\subset G$ (with $G$, $G'$ simple, connected, with trivial centre, no compact factors and such that the rank of $G$ is $>1$) with Zariski dense image, extends to a homomorphism ${\hat{\phi}} : G\rightarrow G'$.

%For hermitian symmetric spaces of higher rank, Margulis' superrigidity theorem implies that any homomorphism 
%$\phi: \Gamma\rightarrow G'$ (with $G'$, connected with trivial centre and no compact factors) with Zariski dense image extends to a homomorphism ${\hat{\phi}} : G\rightarrow G'$.

In rank one, it is well known that superrigidity fails.  In particular, there exist representations of a given lattice in $\PU(2,1)$ into other Lie groups which do not extend to homomorphisms of the whole group $\PU(2,1)$ (see Remark \ref{sec:MostowEx}).
%A very active study of representations of lattices into hermitian Lie groups originated with Toledo's (\cite{T}) description of maximal representations of lattices in $\PU(1,1)$ into $\PU(2,1)$ (see \cite{P} for a recent overview).  
In this paper we are interested in
% a target Lie group which is not hermitian.  We classify 
representations (up to conjugation) of a class of Deligne-Mostow lattices into $\pgl(3,\C)$.   

An important motivation for our results is the study of complex projective structures. Indeed, given a complex projective structure on a complex surface $M$, one can define a developing map $D: \tilde{M}\rightarrow \C P^2$ and a holonomy representation 
$\rho : \pi_1(M)\rightarrow \pgl(3,\C)$. Complex hyperbolic manifolds carry an induced complete complex projective structure. It is an open question whether a complex hyperbolic manifold has a different complex projective structure than this one. On the other hand, fixing the complex structure of a finite volume complex hyperbolic surface, its complex projective structure is unique by a theorem of Mok-Yeung \cite{MY} (see also \cite{Klingler}).
This means that if a complex hyperbolic manifold carried a complex projective structure different from the one induced by the complex hyperbolic structure, it would have to come from a different complex structure.
It is possible that, for the non-cocompact lattices, some of the representations in our list correspond (taking torsion free subgroups) to complex projective structures on complex manifolds diffeomorphic but not biholomorphic to complex hyperbolic manifolds. One should be aware that in the compact case, Siu's rigidity theorem (valid for any compact hermitian locally symmetric space) implies that any other K\" ahler complex structure on the manifold is biholomorphic or conjugate-biholomorphic to the original complex hyperbolic structure.

A  more comprehensive overview of the panorama of known rigidity theorems that motivate our work and relate to it is given in Section \ref{sec:rigidity}.

In section \ref{results} we state the classification of all representations where the generators we chose are of the same type as the generators of the Deligne-Mostow lattice.

In Section \ref{section:DM} we review the definition of the Deligne-Mostow lattices with 3-fold symmetry in $\PU(2,1)$.
The main result which allows for our computations is the explicit presentation of the 3-fold Deligne-Mostow lattices. It was obtained in several steps culminating in \cite{Irene}, where a unified treatment of all 3-fold Deligne-Mostow lattices is given. We will use a presentation of each Deligne-Mostow group using two generators easily obtained from the presentations in \cite{Irene}.

In section \ref{sec:36} we give the detailed proofs for one case. 
The other cases are calculated in the same way and can be found in the SAGE and MAPLE companion notebooks to this paper. 
The main strategy is to use Gr\"obner basis methods to solve the equations 
imposed by the relators in the presentation. We use computations in SAGE and in MAPLE which are available on GitHub (\cite{GH}).
The case of representations not preserving generator types is treated in the Maple files. 
In these, we make use of the Rational Univariate Representation to get formal parameterizations of the solutions when the system we obtain is 0-dimensional (for an introduction to the method in the context of representations to $\pgl(3,\C)$ see \cite{FKR}). The computations show that the representations obtained for six of the lattices are all locally rigid (that is, 0-dimensional), while for three of them, there exists a 1-dimensional branch of representations.  We were not able to find explicit parametrizations of them but they are given by explicit Gr\"obner basis in the available MAPLE file (see also the comments at the end of Section \ref{results}).  
With the exception of the Deligne-Mostow lattice itself (because of the local rigidity theorem for complex hyperbolic lattices -- see Lemma \ref{lemma:complexification}), we do not know if this local rigidity phenomenon for type preserving representations is specific to the lattices we studied or has a more general scope. With an appropriate conjugation, the representations have values in an algebraic field extension of $\Q$ and we studied the orbits by the Galois group. 
%One important feature is that all representations which preserve a hermitian form of type (2,1) are in the same orbit by the Galois group. By the Mostow-Prasad Theorem they are either non-discrete or the image is not a lattice.

We thank Sorin Dumitrescu for useful discussions and for pointing to us the local rigidity result in \cite{MY}.
The second author also acknowledges the support of the FSMP and of the EPSRC grant EP/T015926/1. The third author acknowledges the support of the COnsejo NAcional de Ciencia Y Tecnolog\'ia, project 769037.
We also thank the anonymous referees for useful comments and remarks.

\section{Rigidity of lattices and projective structures}\label{sec:rigidity}

In this section we give some background on the problem {of finding} representations of a lattice and several fundamental rigidity theorems which are known.

Representations of a  group $\Gamma$ into a Lie group $G$ have been studied for a long time. 
It is common to consider only representations modulo conjugation by the Lie group and refer to  the set of deformations modulo conjugation as the deformation variety.    
For instance, if $\Gamma$ is the fundamental group of a compact surface and $G=\slc(2,\R)$, then one component of the deformation variety is Teichm\"uller space, containing only discrete and faithful representations. 
On the other hand, assume that $G$ is a higher dimensional semi-simple group and that $\Gamma$ is a lattice in a Lie group (that is, a discrete subgroup with a finite volume quotient for the Haar measure). 
Then, building on partial results by Calabi, Selberg, Calabi-Vesentini, the following local rigidity result was obtained by A. Weil in the cocompact case and by Garland-Raghunatan in the non-cocompact case.

\begin{theo}\label{thm:Weil}
Let $\rho : \Gamma\rightarrow G$ be an isomorphism onto its image, which is a lattice in the semi-simple Lie group $G$. 
 If $G$ is not locally isomorphic to $\slc(2,\R)$ or  $\slc(2,\C)$  then the embedding is locally rigid.
In other words, all representations near the embedding are conjugate.
\end{theo} 

A slightly different but strongly related rigidity concept is that of infinitesimal rigidity.
The infinitesimal rigidity of a representation $\rho : \Gamma\rightarrow G$ is, by definition, the vanishing of the cohomology group $H^1(\Gamma, \mathfrak{g})$, where $\Gamma$ acts on the Lie algebra of $G$, $\mathfrak{g}$, by the adjoint representation
$Ad_{\rho} : G\to Aut (\mathfrak{g})$.  
The cohomology group is identified with the {Zariski} tangent space of the deformation variety. 
Then Weil's {criterion} states the following: if $\Gamma$ is finitely generated, then an infinitesimally rigid representation is also locally rigid.  
Note that here we are not assuming that $G$ is semi-simple nor that $\Gamma$ is a lattice.

The opposite implication is in general not true.
It could happen that $H^1(\Gamma, \mathfrak{g})\neq \{ 0\}$, but that the representation is still locally rigid.  
It occurs when the elements of $H^1(\Gamma, \mathfrak{g})\neq \{ 0\}$ are not tangent to a path in the deformation variety.  
But, under the assumptions of Theorem \ref{thm:Weil} above, one can, in fact, prove infinitesimal rigidity.  
In particular, for  a lattice $\Gamma\subset \PU(n,1)$, one has $H^1(\Gamma, \mathfrak{su}(n,1))=0$. 
This implies the following local rigidity property for the induced representation $ \Gamma\to \pgl(n,1)$ obtained by the inclusion $ {\PU}(n,1)\to \pgl(n,1)$ (see \cite{Klingler} prop. 3.2):

\begin{lemma} \label{lemma:complexification} Let $\Gamma\subset \PU(n,1)$ be a lattice and $\iota :\Gamma\to  \pgl(n+1)$ be the representation obtained by composition with the inclusion $ {\PU}(n,1)\to \pgl(n+1)$.  Then $\iota$ is locally rigid.
\end{lemma}

\begin{proof}
By Weil's {criterion} it suffices to prove that the representation is infinitesimaly rigid, that is $H^1(\Gamma, \mathfrak{sl}(n+1,\C))=0$. As 
$\mathfrak{sl}(n+1,\C)= \mathfrak{su}(n,1)\otimes\C=\mathfrak{su}(n,1)\oplus i\mathfrak{su}(n,1)$ one has that
$H^1(\Gamma, \mathfrak{sl}(n+1,\C))=H^1(\Gamma, \mathfrak{su}(n,1))\otimes \C=H^1(\Gamma, \mathfrak{su}(n,1))\oplus iH^1(\Gamma, \mathfrak{su}(n,1))=0$.
\end{proof}

This lemma, with the same proof, is also valid for other representations of lattices obtained through complexifications of Lie groups.  
Related to local rigidity of complex hyperbolic lattices one should mention that the complex structure of a quotient $\Gamma \backslash H^n_\C$ of complex hyperbolic space by a torsion free lattice is locally rigid \cite{CV}.  It is still an open problem to decide if this complex structure is unique.

A step further is the passage from local rigidity to global rigidity in the case of lattices.  
This is known as the Mostow theorem and is due to Mostow for the cocompact case and to Prasad for the non-cocompact case (see \cite{WM} pg. 309):

\begin{theo}
Let $\rho : \Gamma_1\rightarrow \Gamma_2$ be an isomorphism between two lattices in  $G_1$ and $G_2$ (with $G_i$ connected non-compact simple Lie group with trivial center non-isogenous to $\slc(2, \R)$).  
Then $\rho$ extends to an isomorphism between the Lie groups $G_1$ and $G_2$.

\end{theo} 

In particular, in the case of two lattices $\Gamma_1$ and $\Gamma_2$ in $\PU(n,1)$,  $\Gamma_2$ is conjugated to $\Gamma_1$ or to its complex conjugate by an element of $\PU(n,1)$.
%The theorem is valid, more generally, for $G_i$ any simple non-compact Lie groups (non-isogenous to $\slc(2,\R)$) with trivial center.
%Related to this rigidity result is the result \cite{CV}} which states that any K\" ahler manifold homotopic to a compact complex hyperbolic manifold is either biholomorphic or conjugate biholomorphic to it.

The main goal of this paper is to classify representations of lattices in $\PU(2,1)$  into $\pgl(3,\C)$.
By Mostow-Prasad rigidity, they are rigid in $\PU(2,1)$.  
But, as it turns out, they are not rigid in $\pgl(3,\C)$.  This should be compared with the superrigidity theorem by Margulis.  
We state it in a simplified form (see \cite{WM} pg. 323): 
% I took a further simplification of the corollary by assuming that G_1 and G_2 are simple and non-compact
\begin{theo}  Assume $G_1$ and $G_2$ are connected non-compact simple Lie groups with trivial centre.  
Assume that the rank of $G_1$ is at least two.  
Let $\rho : \Gamma\rightarrow G_2$ be a homomorphism from a lattice $\Gamma\subset G_1$ such that $\rho(\Gamma)$ is Zariski dense in $G_2$.  
Then $\rho$ extends to a continuous homomorhism $\hat{\rho} : G_1\to G_2$.

\end{theo}

The theorem is also valid if $G_1$ is a Lie group of rank one other than those isogeneous to $\SU(n,1)$ or $\SO(n,1)$ (see \cite{Cor}).  The techniques used in the rank one case are inspired by Siu's theorem  \cite{Siu} which states that any K\" ahler manifold homotopic to a compact complex hyperbolic manifold is either biholomorphic or conjugate biholomorphic to it.

\begin{rk}\label{sec:MostowEx}

Mostow (see \cite{mostow} section 22) constructed  two cocompact lattices (which are arithmetic) $\Gamma_1$ and $\Gamma_2$ of $\PU(2,1)$ and a surjective homomorphism
$\rho : \Gamma_1\to \Gamma_2$ with infinite kernel.  This gives an example of a violation of superrigidity for lattices in 
$\PU(2,1)$.
We are not able to exactly relate this homomorphism with our representations, since our group {is} the same as Mostow's group only up to finite index.
\end{rk}

\subsection{Real and complex  projective structures.}

In this paper we construct a large class of homomorphisms of Deligne-Mostow lattices in $\PU(2,1)$ into $\pgl(3,\C)$.  
One might think that these representations (taking torsion free subgroups of the lattice) might be holonomy representations of different complex projective structures on the quotient of complex hyperbolic space by the lattice. 
But if the complex structure of the quotient is fixed, that cannot be the case:

\begin{theo}[Mok-Yeung \cite{MY}]

Let $\Gamma\backslash H^2_\C$ be a complex hyperbolic manifold of finite volume. Then its complex projective structure induced by the embedding $H^2_\C\subset \C P^2$ is the unique projective structure compatible with its complex structure.
\end{theo}

On the other hand, the existence of different complex structures on the quotient $\Gamma\backslash H^2_\C$ (for a torsion free lattice $\Gamma$) is an open problem.
The theorem is also true for complex hyperbolic manifolds of dimension greater than two (see \cite{MY}) and for higher rank hermitian symmetric domains (see \cite{Klingler}). 

It is interesting to compare the  analogous situation arising in the case of lattices in $\PO(n,1)$.  
One can consider the embedding $\PO(n,1)\subset \pgl(n+1,\R)$ and look for deformations.  
The argument in Lemma \ref{lemma:complexification} does not apply anymore and indeed in many situations the lattice is not locally rigid in $\pgl(n+1,\R)$ (see \cite{JM}).  
In particular, the existence of totally geodesic hypersurfaces in a quotient of real hyperbolic space, $\Gamma \backslash H^n_\R$ implies the existence of deformations, called bendings, of the group
$\Gamma$ inside $\pgl(n+1,\R)$.  The space of deformations of representations of $\Gamma$ into $\pgl(n+1,\R)$ modulo conjugation has been studied for a long time with contributions by several authors (see the survey \cite{Q} which {describes}, among others, contributions by Kac-Vinberg, Koszul and Benoist).  
In particular, starting from a cocompact discrete subgroup of $\PO(n,1)$, there exists a whole component  of the space of deformations  consisting  of discrete faithful representations $\rho$ of $\Gamma$ such {that} there exists a convex open subset in projective space $\Omega\subset \R P^n$ whose quotient by $\rho(\Gamma)$ is compact.  These deformations correspond to holonomies of real projective structures on the underlying manifold structure of the hyperbolic structure.

\subsection{Other deformations.}

Another problem related to the one we study is to classify embeddings of lattices
 in ${\PU}(n,1)$ into ${\PU}(n+1,1)$ or the analogous problem of embeddings of lattices of
 ${\PO}(n,1)$ into ${\PO}(n+1,1)$.  For the latter, bending deformations were  also studied in \cite{JM}.  The former
  has been studied in \cite{GM} and \cite{T}.  In particular, a discrete, torsion free, cocompact subgroup $\Gamma\subset \PU(n,1)$ embedded in ${\PU}(n+1,1)$ preserving a totally geodesic complex hypersurface in complex hyperbolic space  has some rigidity in the sense that a connected component of the representation space into ${\PU}(n+1,1)$
  consists also of discrete representations preserving a totally geodesic complex hypersurface.  For a review and  generalizations of this phenomenon we refer to \cite{KM}.  See also \cite{P} for a recent overview of maximal representations of lattices in $\PU(1,1)$ into $\PU(2,1)$.

\section{The results}\label{results}

We refer to section \ref{section:DM} for the definition of Deligne-Mostow lattices. For the statement of the results we simply recall that 3-fold Deligne-Mostow lattices of type one are parametrised by a pair of integers $(p,k)$ and are generated by two elements $J$ and $R_1$ which are, respectively, a regular elliptic element of order 3 and a complex reflection of order $p$.
We say that a representation of the  lattice preserves generators types if the image of a complex reflection generator (respectively regular elliptic) is a complex reflection (respectively regular elliptic).  Note that in this definition the image of a complex reflection could be the identity and {its order} could be a a factor of the order of the Deligne-Mostow generator.  That is, conjugacy classes of reflections might not  be preserved.

\begin{theo} Classification of representations of type one Deligne-Mostow lattices (with  3-fold symmetry)  into $\pgl(3,\C)$:
\begin{enumerate}
\item Representations preserving generator types are locally rigid and are classified in Tables \ref{Tb:CompactDeligneMostow} and \ref{Tb:NonCompactDeligneMostow}.
\item
Representations without constraints on generators are also locally rigid, except for lattices (4,3), (4,4) and  (6,2) which have one dimensional branches for certain representations with regular elliptic generators.  They are enumerated in the MAPLE file in \cite{GH}.
\end{enumerate}
\end{theo}

In Table \ref{Tb:CompactDeligneMostow} are the details for each compact 3-fold type one Deligne-Mostow lattice. The representations are defined with coefficients in a cyclotomic extension of $\Q$, which depends on $(p,k)$. We also compute the orbits of the Galois group action on the space of representations. For the compact Deligne-Mostow lattices, there exists an invariant Hermitian form which is always non-degenerate.

\begin{table}[h!]
\centering
\begin{tabular}{|c|c|p{1.5cm}|c|c|c|}
\hline
$(p,k)$ & Total & Galois Orbits  & $\Q-$extension & Irreducible &Reducible\\ \hline

\multirow{2}{2em}{(3,4)} &\multirow{2}{1em}{18}& \multirow{2}{4em}{2} &{$\mathbf{\Q(\zeta_{36})}$} & 1 & 0 \\ \cline{4-6}
& & & {$\Q(\zeta_{18})$}& 1 &0 \\ \hline

(3,5) & 24 & 1 &{$\mathbf{\Q(\zeta_{45})}$} & 1 & 0 \\ \hline

\multirow{2}{2em}{(4,3)} & \multirow{2}{1em}{20} &\multirow{2}{2em}{5} & ${\mathbf{\Q(\zeta_{36})}}$& 3 & 0  \\ \cline{4-6}
& &  & {$\mathbf{\Q}$}& 0&(0,2)\\ \hline

\multirow{2}{2em}{(5,2)} & \multirow{2}{1em}{12} & \multirow{2}{1em}{2} & $\mathbf{\Q(\zeta_{15})}$ & 1 & 0\\ \cline{4-6}
& & & $\Q(\zeta_{10})$ &  1 & 0 \\\hline

\multirow{2}{2em}{(5,3)} & \multirow{2}{1em}{25} & \multirow{2}{1em}{2} & $\mathbf{\Q(\zeta_{45})}$& 1 & 0 \\ \cline{4-6}
& &  & $\Q$ & 0 &(0,1)  \\ \hline
\end{tabular}
\caption{Compact 3-fold Deligne-Mostow Lattices}
\label{Tb:CompactDeligneMostow}
\end{table}

In Table \ref{Tb:NonCompactDeligneMostow} are the details on non-compact 3-fold type one lattices.
Each pair $(p,k)$ defines a $\Q$-extension and each representation of the $(p,k)$ lattice has coefficient in (a subfield of) this $\Q-$extension.
For each representation there exists an invariant Hermitian form that can be non-degenerate with signatures (2,1) or (3,0), or degenerate. While the irreducible representations can only come from non-degenerate configurations, the reducible ones can come from either degenerate or non-degenerate configurations (see Section \ref{sec:36}). In the Reducible column the numbers $(m,n)$ mean that $m$ reducible representations come from non-degenerate configurations and $n$ come from degenerate ones. The last column corresponds to the number of representations which contain in their kernel an element of the centraliser of the cusp group. One may interpret those representations as representations which factor through the orbifold fundamental group of a compactification of the complex hyperbolic orbifold defined by the Deligne-Mostow lattice.

\begin{table}[h!!]
\centering
\begin{tabular}{|p{1cm}|p{1cm}|p{1.5cm}|c|c|c|p{2cm}|c|}
\hline
$(p,k)$ & Total & Galois Orbits &$\Q-$extension & Irreducible &Reducible & Hermitian Form & Factors\\ \hline
\multirow{4}{4em}{(3,6)}  & \multirow{4}{4em}{39}& \multirow{4}{4em}{14}&\multirow{2}{2em}{$\mathbf{\Q(\zeta_9)}$}& 1 & 0 &(2,1) & 0 \\ \cline{5-8}
& & & & 2& 0& (3,0)& 2\\ \cline{4-8}
 & & &$\Q(\zeta_3)$& 0& (6,4) & Degenerate &  10 \\ \cline{4-8}
 & & &$\Q$& 0& (0,1) & Degenerate &1 \\ \hline

\multirow{8}{4em}{(4,4)} & \multirow{8}{4em}{26}&\multirow{8}{4em}{13} &\multirow{3}{2.5em}{$\Q(\zeta_{12})$}& 1 & 0 &(2,1)& 0 \\ \cline{5-8}
& & & & 1 & 0 & (3,0) & 1 \\ \cline{5-8}
& & & & 0 & (2,2) &  Degenerate & 4 \\ \cline{4-8}
& & & $\Q(\zeta_6)$ & 2 & 0 & (3,0) &2 \\ \cline{4-8}
& & & \multirow{2}{2.5em}{$\mathbf{\Q(\zeta_4)}$} & 1 & 0 & (2,1) & 0 \\ \cline{5-8}
& & & & 0 & (2,0) & Degenerate & 2\\ \cline{4-8}
& & & $\Q$ & 1 & (0,1) & Degenerate & 2 \\ \hline 

\multirow{5}{4em}{(6,2)} & \multirow{5}{4em}{13}&\multirow{5}{4em}{5} & \multirow{2}{2.5em}{$\mathbf{\Q(\zeta_{18})}$} & 1 & 0 & (2,1) & 0 \\ \cline{5-8}
& & & & 1 & 0 & (3,0) & 1\\ \cline{4-8}
& & & $\Q(\zeta_6)$ & 1 & 0 & (3,0) & 1 \\ \cline{4-8}
& & & \multirow{2}{1.5em}{$\Q$} & 1 & 0 & (3,0) & 1\\ \cline{5-8}
& & & & 0 & (0,1) & Degenerate & 1 \\\hline

\multirow{4}{4em}{(6,3)} & \multirow{4}{4em}{44}& \multirow{4}{4em}{19}&\multirow{2}{2.5em}{$\mathbf{\Q(\zeta_9)}$} & 1 & 0 & (2,1)& 0 \\ \cline{5-8}
& & & & 2 & 0 & (3,0) & 2 \\ \cline{4-8}
%& & & & 1 & 1 & Degenerate &  2\\ \cline{4-8}
& & & $\Q(\zeta_3)$ & 0 & (6,8)& Degenerate & 12\\ \cline{4-8}
& & & $\Q$ & 0 &(0,2) & Degenerate & 2 \\ \hline

\end{tabular}
\caption{Non-compact 3-fold Deligne-Mostow Lattices}
\label{Tb:NonCompactDeligneMostow}
\end{table}

%\newpage

Concerning the representations which do not preserve types of generators, 
the computations show that they  are all locally  rigid except for one dimensional branches for Deligne-Mostow lattices $(4,3), (4,4)$ and $(6,2)$.
We did not obtain, though, a parametrisation of the few examples having a positive dimension.
{In fact, we obtain the Hilbert dimension of the system of equations describing representations for each fixed order of the eigenvalues of $R_1$ (note that,  for each Deligne-Mostow group of of type $(p,k)$ the type preserving representations, $R_1$ has eigenvalues 1 and order $p$).  In a few cases, described in detail in the MAPLE file, the Hilbert dimension is one. This implies that there exists a one dimensional branch without the need of computing a parametrization.  We do not have a general understanting of this phenomenon.}

We also determine which representations may be lifted to $\gl(3,\C)$ representations (see \cite{GH}).  The representations with generators $\rho(J)$ and $\rho(R_1)$ of other types are not tabulated.

In the following sections, we give details for the calculations in the case of the Deligne-Mostow lattice $(3,6)$. 
We also compute the reducible representations of this lattice when the generators are \emph{not} of the same type as the ones of the Deligne-Mostow lattice (note that in this case, all representations are locally rigid). 
The same calculations for the other Deligne-Mostow lattices in the table can be found on GitHub (\cite{GH}).  

\section{Deligne-Mostow lattices}\label{section:DM}
The Deligne-Mostow lattices were introduced and studied by Mostow and by Deligne and Mostow in various works, including \cite{mostow} and \cite{delignemostow}. 
They have a long history dating back to by Picard, Le Vavasseur and Lauricella. 
Deligne and Mostow start with a \emph{ball N-tuple} $\mu=(\mu_1, \dots, \mu_N)$ which is a set of $N$ real numbers in $(0,1)$ whose sum is 2.
Then they look at a \emph{hypergeometric function} defined using these parameters. 
Finally, they use the \emph{monodromy action} to build some lattices in $PU(N-3,1)$. 
This list of lattices contains the first known examples {of} non-arithmetic complex hyperbolic lattices in dimensions 2 and 3. 

Here we will concentrate on a special class of the Deligne-Mostow lattices in $\PU(2,1)$, called the lattices with 3-fold symmetry (or the 3-fold Deligne-Mostow lattices) and of type one. 
More precisely, first we restrict to the case of dimension 2. 
This means that we are looking at the group $\PU(2,1)$ acting on 2-dimensional complex hyperbolic space and hence that we are looking at ball 5-tuples. 
In this dimension, the finite list of lattices that Deligne and Mostow found contains only lattices with some special symmetry. 
These are called the 2-fold and 3-fold symmetry lattices. 
A lattice has $m$-fold symmetry if $m$ of the $N$ elements of the ball $N$-tuple are equal. 

In this work we will only look at some of the lattices with 3-fold symmetry. 
For all of those, one can find a construction for a fundamental domain and a presentation in \cite{Irene}.
In principle the 2-dimensional Deligne-Mostow lattices depend on the 5 elements of the ball 5-tuple. 
The condition on the sum of the elements of the ball 5-tuple reduces the parameters to 4. 
Of these 4, 3 are equal, so the lattices only depend on 2 parameters. 
One common choice for the parameters is to choose the orders of some of the generators. 
We will hence identify the lattices using a pair $(p,k)$ which corresponds to the ball 5-tuple $(1/2-1/p, 1/2-1/p, 1/2-1/p, 1/2+1/p-1/k, 2/p+1/k)$. 

The Deligne-Mostow lattices in $\PU(2,1)$ with 3-fold symmetry are divided in 4 \emph{types} according to the ranges of $p$ and $k$. 
The type determines the presentation, the volume formula and the combinatorics of a fundamental domain. 
Here we will look at the lattices \emph{of type one}, for which an explicit fundamental domain and presentation can be found, for example, in \cite{boadiparker}.
We are choosing these because they have minimal complexity (minimum number of facets in a fundamental domain, shortest presentation). 
Note that they are all arithmetic.
In terms of $p$ and $k$ they are characterised by 
\begin{align*}
0<p &\leq 6, & k \leq \frac{2p}{p-2}.
\end{align*}
Then a presentation for them is given by 
\begin{equation}\label{eq:Pres}
\Gamma=\left\langle J,P,R_1,R_2 \colon
\begin{array}{l l}
J^3=R_1^p=R_2^p=(P^{-1}J)^k=I, \\
R_2=P R_1P^{-1}=JR_1J^{-1}, \ P=R_1R_2
\end{array} \right\rangle,
\end{equation}
%We will denote the Hermitian form preserved by the matrices in $PU(2,1)$ as $H_{DM}$.
We can rewrite this presentation in terms of the two generators $J$ and $R_1$ in the following way: 
\begin{equation}\label{eq:Pres2gen}
\Gamma=\left\langle J,R_1 \colon
\begin{array}{l l}
J^3=R_1^p=(R_1J)^{2k}=R_1JR_1J^2R_1JR_1^{p-1}J^2R_1^{p-1}JR_1^{p-1}J^2=Id \\
\end{array} \right\rangle,
\end{equation}
One can go back to the previous presentation writing $R_2=JR_1 J^{-1}$ and $P=R_1R_2$.

For the Deligne-Mostow lattices, the generator $J$ is always regular elliptic of order 3, while the generator $R_1$ is a complex reflection of order $p$. 
Recall that by the classification of isometries in complex hyperbolic space (see, for example, \cite{CG}), elliptic elements have at least one fixed point inside $H^2_\C$, parabolic elements have a unique fixed point on the boundary and loxodromic elements have exactly two fixed points on the boundary. 
Elliptic elements have finite order and have all eigenvalues of norm 1. 
They can be of two types: regular elliptic when they have three distinct eigenvalues, or complex reflections when they have a repeated eigenvalue.

The same analysis can be carried out for the Deligne-Mostow lattices of the other three types. 
This can be found in the note \cite{PARTII}.

%{\color{red}This part we need if we decide to keep the type of the generators}
%For the Deligne-Mostow lattices, the generator $J$ is always regular elliptic, while the generator $R_1$ is a complex reflection. 
%We will only look at representations that preserve the types of the generators. 

%A fundamental domain for these lattices as constructed first in \cite{boadiparker}, has six vertices. 
%Here, for simplicity, we will choose four of them and use the 4-transitivity in $\C\P^2$ to send them to the following points 
%\begin{align}\label{eq:Pts}
%\zz 3 &= 
%\begin{bmatrix}
%1 \\ 1 \\ 1
%\end{bmatrix}, &
%\zz 4 &= 
%\begin{bmatrix}
%1\\ 0 \\ 0
%\end{bmatrix}, &
%\zz 5 &= 
%\begin{bmatrix}
%0 \\ 1 \\ 0
%\end{bmatrix}, &
%\zz 6 &= 
%\begin{bmatrix}
%0 \\ 0 \\ 1
%\end{bmatrix}.
%\end{align}
%
%The action of the generators on these vertices is as follows: 
%\begin{align}\label{eq:ActionGen}
%&J \colon & &P \colon & &R_1 \colon & &R_2 \colon \\
%\zz4&\mapsto \zz5 & \zz3&\mapsto \zz3 & 
%\zz3&\mapsto \zz3 & \zz3&\mapsto \zz3 \nonumber\\ 
%\zz5&\mapsto \zz6 & \zz5&\mapsto \zz6 & 
%\zz4&\mapsto \zz4 & \zz5&\mapsto \zz5 \nonumber\\ 
%\zz6&\mapsto \zz4 & \zz6&\mapsto \zz4 & 
%\zz5&\mapsto \zz6 & \zz6&\mapsto \zz4 \nonumber\\ 
%\end{align}
\section{ The group (3,6)}\label{sec:36}

In this section we look in detail at the case $(p,k)=(3,6)$.
%\begin{equation}\label{eq:Pres(3,6)}
%\Gamma=\left\langle J,P,R_1,R_2 \colon
%\begin{array}{l l}
%J^3=R_1^3=R_2^3=(P^{-1}J)^6=I, \\
%R_2=P R_1P^{-1}=JR_1J^{-1}, \ P=R_1R_2
%\end{array} \right\rangle,
%\end{equation}
%
%One can also give a presentation with two generators as
The presentation for this case is: 
\begin{equation}\label{eq:RPres(3,6)}
\Gamma=\left\langle J,R_1 \colon
\begin{array}{l l}
J^3=R_1^3=(R_1J)^{12}=R_1JR_1J^2R_1JR_1^2J^2R_1^2JR_1^2J^2=Id \\
\end{array} \right\rangle,
\end{equation}

We first note that for the Deligne-Mostow lattices, one has that $J$ is a regular elliptic element, while $R_1$ is a complex reflection. 
In general, while the presentation guarantees that $\rho(J)$ and $\rho(R_1)$ are both elliptic, we cannot say more about their type. 
We will hence look at all possible combinations of types: 
\begin{itemize}
\item when $\rho(J)$ is regular elliptic and $\rho(R_1)$ is a complex reflection, 
\item when $\rho(J)$ is a complex reflection and $\rho(R_1)$ is regular elliptic,
\item when $\rho(J)$ and $\rho(R_1)$ are both regular elliptic, 
\item when $\rho(J)$ and $\rho(R_1)$ are both complex reflections. 
\end{itemize}

Note that if $\rho(J)$ and $\rho(R_1)$ are both complex reflections, then the group preserves the intersection of two complex lines in $\C P^2$. The group is therefore reducible and we will hence ignore this last case.

%%%%%%%%%%%%%%%%%%%%%%%%%%%%%%%

\subsubsection{$\rho(J)$ is regular elliptic and $\rho(R_1)$ is a complex reflection}

First, let us assume that the fixed points of $\rho(J)$ do not intersect the fixed points of $\rho(R_1)$. 
We will call these the \emph{non-degenerate configurations}.
We remark that this case might give rise to reducible representations when there exists a degenerate Hermitian form preserved by the generators with non-trivial kernel.
We note in the following $\omega= -\frac{1}{2}+\frac{\sqrt{3}}{2}i$ and $\zeta_9$ a primitive 9th-root of unity.

\begin{prop}\label{prop:Parameters36}
A representation $\rho: \Gamma \rightarrow \pgl(3,\C)$ such that $\rho(R_1)$ is a complex reflection and $\rho(J)$ has three distinct eigenvectors which do not intersect the fixed line of $\rho(R_1)$ are given, up to conjugation, by the matrices
\begin{align}
J&= 
\begin{bmatrix}
0 & 0 & 1 \\
-1 & 0 & 0 \\
0 & 1 & 0
\end{bmatrix}
\end{align}
and 
\begin{align}
\rho(R_1)&= 
\begin{bmatrix}
1 & -r_1 & r_1 \\
0 & 1-r_2 & r_2 \\
0 & 1-r_2-\omega & r_2+\omega
\end{bmatrix},
\end{align} 
 with $(r_1,r_2)$ 
 in the following list

 \begin{eqnarray*}
\alpha_1 &= &( 0,1),\\
\alpha_2 &= &(\omega-1,1),\\
\alpha_3 &= &( \frac{1}{2}(\omega-1),\frac{1}{2}(\omega+3)),\\
\alpha_4 &= &( \frac{1}{2}(\omega-1),\frac{-1}{2}(3\omega+1)),\\
\alpha_5 &= &( 0,-\omega),\\
\alpha_6 &= &(\omega-1,-\omega),\\
\alpha_7 & = & \left( \frac{1}{2}(\zeta_9^3+\zeta_9^2-\zeta_9-1), \frac{1}{2}(-\zeta_9^3+\zeta_9^2+\zeta_9+1)\right),\\
\alpha_8 & = & \left( \frac{1}{2}(\zeta_9^5+\zeta_9^4+\zeta_9^3+\zeta_9-1), \frac{1}{2}(\zeta_9^5-\zeta_9^4-\zeta_9^3-\zeta_9+1)\right),\\
\alpha_9 & = & \left( \frac{1}{2}(-\zeta_9^5-\zeta_9^4+\zeta_9^3-\zeta_9^2-1),\frac{1}{2}(-\zeta_9^5+\zeta_9^4-\zeta_9^3-\zeta_9^2+1)\right),\\
\alpha_{10} & = & \left( \frac{1}{2}(\zeta_9^3+\zeta_9^2-\zeta_9-1), \frac{1}{2}(-\zeta_9^3-\zeta_9^2-\zeta_9+1)\right),\\
\alpha_{11} & = & \left(\frac{1}{2}(-\zeta_9^5-\zeta_9^4+\zeta_9^3-\zeta_9^2-1), \frac{1}{2}(\zeta_9^5-\zeta_9^4-\zeta_9^3+\zeta_9^2+1)\right),\\
\alpha_{12} & = & \left( \frac{1}{2}(\zeta_9^5+\zeta_9^4+\zeta_9^3+\zeta_9-1), \frac{1}{2}(-\zeta_9^5+\zeta_9^4-\zeta_9^3+\zeta_9+1)\right),\\
\alpha_{13} & = & \left( \frac{1}{2}(\zeta_9^3-\zeta_9^2+\zeta_9-1), \frac{1}{2}(-\zeta_9^3-\zeta_9^2-\zeta_9+1)\right),\\
\alpha_{14} & = & \left( \frac{1}{2}(\zeta_9^5+\zeta_9^4+\zeta_9^3+\zeta_9^2-1), \frac{1}{2}(\zeta_9^5-\zeta_9^4-\zeta_9^3+\zeta_9^2+1)\right),\\
\alpha_{15} & = & \left( \frac{1}{2}(-\zeta_9^5-\zeta_9^4+\zeta_9^3-\zeta_9-1), \frac{1}{2}(-\zeta_9^5+\zeta_9^4-\zeta_9^3+\zeta_9+1)\right).
\end{eqnarray*}

\end{prop}

%It is interesting to note that only the first six representations can be lifted to $\gl(3,\C)$ representations (see \cite{GH}).

\begin{proof}

The projective group acts transitively on ordered sets of three distinct points and a line not containing them in $\C P^2$.
One observes that the fixed points of $J$ are $[1,1,-1]$, $[ \omega, \bar{\omega}, 1]$, $[ \bar{\omega}, \omega, 1]$ and the line fixed by $R_1$ (generated by $[1,0,0]$ and $[1,1,1]$) does not meet the fixed points. 
 
 Then the matrix of $\rho(J)$ may be fixed to be the one in the statement of the lemma and the matrix of $\rho(R_1)$ will be of the form
 \begin{align}
\rho(R_1)&= 
\begin{bmatrix}
1 & -r_1 & r_1 \\
0 & 1-r_2 & r_2 \\
0 & 1-r_3 & r_3
\end{bmatrix}.
\end{align} 
By a computation, $\rho(R_1)^3=Id$ if and only if $r_3=r_2+x$ with $x^2+x+1=0$.
%The proof now is a computation using the Gr\"obner basis package in  SAGE.  

In fact, the first six solutions appear from a simple factorisation of the equations.  The others come from the following system:
 {\tiny{
$$
(8r_2^3+6(i\sqrt{3}-3)r_2^2-12 (i \sqrt{3}-1)r_2+6i\sqrt{3}-1)(8r_2^3+6(i\sqrt{3}-3)r_2^2-12 (i \sqrt{3}-1)r_2+6i\sqrt{3}+1)=0,
$$
$$
(32i r_2^4\sqrt{3}-96( i+1) r_2^3\sqrt{3}+(92 i +276)\sqrt{3}r_2^2-25i \sqrt{3}-12r_1-264r_2+75)( 8r_2^3+6(i\sqrt{3}-3)r_2^2-12 (i \sqrt{3}-1)r_2+6i\sqrt{3}-1)=0,
$$
$$
%\frac{3}{2}r_1-\frac{1}{2}i r_1\sqrt{3}+\frac{3}{2}r_2-\frac{1}{2}i r_2\sqrt{3}+r_1^2-r_2^2-\frac{1}{2}+\frac{1}{2}\sqrt{3}=0
({3}-i\sqrt{3})r_1 +({3}-i\sqrt{3})r_2+2r_1^2-2r_2^2-1+\sqrt{3}=0.
$$
}
}

Writing

\begin{eqnarray*}
p_1(r_2)&=&8r_2^3+6(i\sqrt{3}-3)r_2^2-12 (i \sqrt{3}-1)r_2+6i\sqrt{3}-1,\\
p_2(r_2)&=&8r_2^3+6(i\sqrt{3}-3)r_2^2-12 (i \sqrt{3}-1)r_2+6i\sqrt{3}+1,
\end{eqnarray*}
both polynomials factor over $\Q(\zeta_9)$. For each root of $p_1,$ the corresponding values of $r_1$ satisfy a quadratic polynomial equation which factors over $\Q(\zeta_9)$. In a similar way, for each root of $p_2,$ the corresponding values of $r_1$ satisfy a linear polynomial equation whose root belong to $\Q(\zeta_9)$. Therefore the whole set of solutions belongs to the field $\Q(\zeta_9)$.   The solutions are then written in the basis of $\Q(\zeta_9)$.
 \end{proof}

\begin{rk}
\begin{itemize}
\item The general procedure in the SAGE notebooks goes as follows.
With the group relations we produce a system of equations and using SAGE we construct a polynomial ideal over $\mathbb{Q}[r_1,r_2,x]$. Then we produce a Gr\"oebner basis for our system of equations. This has, as output, an array of polynomials that generates our initial ideal. 
Finally, we solve these polynomials for each variable to find the solutions.

\item For each solution $\alpha_j,$ we take lifts into $\slc(3,\C)$ of the form: 
\[k\cdot J \quad \mbox{and} \quad l\cdot R_1(\alpha_j).\]

In order to have a lift of the representation, the previous matrices should satisfy the group relations. Therefore, we use the relations to obtain a system of equations in the variables $k$ and $l.$ It is interesting to note that in the (3,6) case only the first six representations can be lifted to $\gl(3,\C)$ representations (see \cite{GH}).

\end{itemize}
\end{rk}

We now look at the \emph{degenerate configurations}, where one eigenvector of $\rho(J)$ intersects the fixed complex line of $\rho(R_1)$.
This gives rise to reducible representations, but note that they might not be completely reducible.

\begin{prop}
The representation $\rho: \Gamma \rightarrow \pgl(3,\C)$ such that $\rho(R_1)$ is a complex reflection and $\rho(J)$ has three distinct eigenvectors with one intersecting the fixed line of $\rho(R_1)$ are given, up to conjugation, by the matrices
\begin{align}
\rho(J)&= 
\begin{bmatrix}
1 & 0 & 0 \\
0 &e^{2i\pi/3} & 0 \\
0 & 0 & e^{4i\pi/3}
\end{bmatrix}
\end{align}
and 
\begin{align}
\rho(R_1)&= 
\begin{bmatrix}
1 & -1 & 1 \\
0 & 1-r_2 & r_2 \\
0 & 1-r_2-x & r_2+x
\end{bmatrix}
&&
or
&
\label{Eq:BlockR1}
\rho(R_1)&= 
\begin{bmatrix}
1 & 0 & 0 \\
0 & 1-r_2 & r_2 \\
0 & 1-r_2-x & r_2+x
\end{bmatrix},
\end{align} 
where $x^2+x+1=0$ and with $r_2$ determined as follows:
 \begin{enumerate}
 \item for $x=e^{2i\pi/3}$, $r_2= 1-\sqrt{3}i/3$ or $r_2= 1/2-\sqrt{3}i/6$,

 \item for $x=e^{4i\pi/3}$, $r_2= 1+\sqrt{3}i/3$ or $r_2= 1/2+\sqrt{3}i/6$.
 
\end{enumerate}

\end{prop}

\begin{proof}
We choose a diagonal form for $\rho(J)$ and impose that $\rho(R_1)$ has the same eigenvector $e_1$. One can then impose that the fixed line in projective space passes through $[1,0,0]$ and $[1,1,1]$. By conjugating by a diagonal matrix, one can suppose that $\rho(R_1)$ is as above with $r_2$ to be determined. As in the previous proposition, a computation using Gr\"oebner basis package in  SAGE gives the result. Note that the two sets of solutions are complex conjugates.

\end{proof}

{
\begin{rk}
Note that for solutions with the second form of $R_1$ in \eqref{Eq:BlockR1}, the image of the group $\langle J, R_1\rangle$ is contained in a copy of $\slc(2,\C)$ inside $\psl(3,\C).$ This implies that, for each solution, the group is contained in a $\psl(2,\C)-$representation of the triangle group $(3,3,12).$ In general, for the $(p,k)$- lattice, we can say that the image is contained in a $\psl(2,\C)-$representation of a triangle group $(3,p,2k).$

Note that if we remove the condition $x^2+x+1=0$ in the previous proposition, there exists a unique (extra) solution for which $R_1=Id.$ This solution is part of the representations with type-preserving generators, and whose image is a cyclic group of order three generated by $J.$ For the lattice (3,6), this is the unique reducible representation with finite image. This unique representation is contained also in representations for lattices (4,3), (5,3), (6,2), and (6,3). This solution is already taken into account in Tables \ref{Tb:CompactDeligneMostow} and \ref{Tb:NonCompactDeligneMostow}.
\end{rk}
}

\subsubsection{$\rho(J)$ is a complex reflection and $\rho(R_1)$ is regular elliptic}

Interchanging the type of $\rho(J)$ and $\rho(R_1)$, a computation with Gr\"obner basis proves the following:

\begin{prop}\label{prop:inverted}
There are no representations $\rho: \Gamma \rightarrow \pgl(3,\C)$ such that $\rho(J)$ is a complex reflection and $\rho(R_1)$ has three distinct eigenvectors which do not intersect the fixed line of $\rho(R_1)$.
\end{prop}

\subsubsection{$\rho(J)$ and $\rho(R_1)$ are regular elliptic}

The final case to consider occurs when both $\rho(J)$ and $\rho(R_1)$ are regular elliptic. A simple computation shows that the two regular elliptic elements cannot have the same eigenspaces.

\begin{prop} The representations $\rho: \Gamma \rightarrow \pgl(3,\C)$ such that $\rho(R_1)$ and $\rho(J)$ are regular elliptic with at least one distinct eigenspace are given, up to conjugation, by the matrices
\begin{align}
J&= 
\begin{bmatrix}
0 & 0 & 1 \\
-1 & 0 & 0 \\
0 & 1 & 0
\end{bmatrix}
\end{align}
and 
\begin{align}
\rho(R_1)&= 
\begin{bmatrix}
1 & s_1 & r_1 \\
0 & s_2 & r_2 \\
0 & s_3 & -1-s_2
\end{bmatrix},
\end{align} 
 with $r_1$, $r_2$, $s_1$, $s_2$, $s_3$ 
 in the following list (up to complex conjugation)
 %(where we denote by $z=\frac{5+i\sqrt{3}}{14}$a root of $7z^2-5z+1=0$, $w=\frac{2+i\sqrt{3}}{7}$ a root of $7z^2-4z+1$, $u=\frac{1+3i\sqrt{3}}{14}$ a root of $7z^2-z+1$, $v=\frac{-1+3i\sqrt{3}}{14}$ a root of $7z^2+z+1$, $s=\frac{-2+i\sqrt{3}}{7}$ a root of $7z^2+4z+1$ and $t=\frac{-5+i\sqrt{3}}{14}$a root of $7z^2+5z+1=0$).
 
 \begin{enumerate}

 \item $r_1 = 0$, $r_2 = -1$, $s_1 = -1$, $s_2 = -1$, $s_3 = 1$,
\item $r_1 = 1$, $r_2 = 1$, $s_1 = 0$, $s_2 = 0$, $s_3 = -1$,
\item $r_1 = 0$, $r_2 = -\omega$, $s_1 = \omega$, $s_2 = -1$, $s_3 = \bar{\omega}$,
 \item $r_1=0$, $r_2=0,s_1=0$, $s_2=\omega$, $s_3=0$,
\item $r_1 = 6/7+4/7i\sqrt{3}$, $r_2 = -4/7+2/7i\sqrt{3}$, $s_1 = -4/7+2/7i\sqrt(3)$, $s_2 = -11/14-5/14i\sqrt{3}$, $s_3 = 1/7+3/7i\sqrt{3}$,
\item $r_1 = 4/7-2/7i\sqrt{3}$, $r_2 = 1/7+3/7i\sqrt{3}$, $s_1 = -6/7-4/7i\sqrt(3)$, $s_2 = -3/14+5/14i\sqrt{3}$, $s_3 = -4/7+2/7i\sqrt{3}$,
\item $r_1 = \omega$, $r_2 = \bar{\omega}$, $s_1 = 0$, $s_2 = 0$, $s_3 = -\omega$,
\item $r_1 = 3/7-5/7i\sqrt{3}$, $r_2 = -1/7-3/7i\sqrt{3}$, $s_1 = -1/7-3/7i\sqrt(3)$, $s_2 = -11/14-5/14i\sqrt{3}$, $s_3 = 4/7-2/7i\sqrt{3}$,
\item $r_1 = -9/7+1/7i\sqrt{3}$, $r_2 = 5/7+1/7i\sqrt{3}$, $s_1 = 5/7+1/7i\sqrt(3)$, $s_2 = -11/14-5/14i\sqrt{3}$, $s_3 = -5/7-1/7i\sqrt{3}$,
\item $r_1 = 1/7+3/7i\sqrt{3}$, $r_2 = 4/7-2/7i\sqrt{3}$, $s_1 = -3/7+5/7i\sqrt(3)$, $s_2 = -3/14+5/14i\sqrt{3}$, $s_3 = -1/7-3/7i\sqrt{3}$,
\item $r_1 = -5/7+1/7i\sqrt{3}$, $r_2 = -5/7+1/7i\sqrt{3}$, $s_1 = 9/7+1/7i\sqrt(3)$, $s_2 = -3/14-5/14i\sqrt{3}$, $s_3 = 5/7-1/7i\sqrt{3}$,

% \item $r_1 = 0, r_2 = -\omega, s_1 = \omega, s_2 = -1, s_3 = \omega-1$ 
 %\item $r_1=0,r_2=0,s_1=0,s_2=\omega,s_3=0$
 %\item $r_1 = 0, r_2 = -1, s_1 = -1, s_2 = -1, s_3 = 1$
 %\item $r_1 = 0, r_2 = -\omega, s_1 = \omega, s_2 = -1, s_3 = \omega-1$
 %\item $r_1 = -1-\omega, r_2 = \omega, s_1 = 0, s_2 = 0, s_3 = 1+\omega$
 %\item $r_1 = -2+2z, r_2 = 2z, s_1 = 2z, s_2 = 1-5z, s_3 = -2z$
 %\item $r_1 = 1-3w, r_2 = 2w, s_1 = 1-5w, s_2 = 1/2-(5/2)w, s_3 = -1+3w$
 %\item $r_1 = 2/3-(4/3)u, r_2 = 2u, s_1 = -2/3-(8/3)u, s_2 = -1/3+(5/3)u, s_3 = -2/3+(4/3)u$
 %\item $r_1 = 2/3+(10/3)v, r_2 = 2v, s_1 = 2v, s_2 = -2/3+(5/3)v, s_3 = 2/3+(4/3)v$
 %\item $r_1 = 2+4s, r_2 = 2s, s_1 = 2s, s_2 = -3/2-(5/2)s, s_3 = 1+3s$
 % \item $r_1 = 2t, r_2 = 2t, s_1 = 2 +2t, s_2 = -2-5t, s_3 = -2t$

\end{enumerate}

\end{prop}

\begin{proof}
We can again fix $J$ to be of the form above and one eigenvector for $R_1$ to be the first column vector. Observe that a regular elliptic element of order three has null trace and, therefore, we obtain the form of the matrix for $R_1$.
We obtain the result finding a Gr\"obner basis for the system of equations as in the previous proposition.
\end{proof}

\subsection{The Galois Conjugation}

Consider the $\Q-$extension by the cyclotomic polynomial $x^6+x^3+1,$ that we will denote $\Q(\zeta_9)$ where $\zeta_9$ is again the 9-th root of unity. This extension has a subfield isomorphic to the $\Q-$extension by $x^2+x+1,$ that we will denote by $\Q(\omega)$ where we recall that $\omega=\zeta_9^3$. Recall also that the Galois group of $\Q(\zeta_9),$ $\Gal(\Q(\zeta_9)/\Q),$ is isomorphic $\Z_6$.

{
\begin{rk}
For the first six solutions in Proposition \ref{prop:Parameters36}, the group image belongs to $\pgl(3,\mathcal{O}_3)$ up to a projective representative, therefore the corresponding groups have discrete image. This follows from the fact that $\mathcal{O}_3$ is a discrete set in $\C.$

For the other solutions, up to a projective representative, the image belongs to $\pgl(3,\mathcal{O}_9)$ where $\mathcal{O}_9$ denotes the ring of integers of $\Q(\zeta_9)$. Unfortunately, we do not know if the image of these groups is discrete.

In general, for every representation of the $(p,k)-$lattice that preserves generators types, we have (up to a projective representative) that its image is a subgroup of $\pgl(3,\mathcal{O}_j)$ where $\mathcal{O}_j$ is the ring of integers of the extension $\Q(\zeta_j)$ that contains the solution. 
\end{rk}}

The action of the Galois group $\Gal(\Q(\zeta_9)/\Q)$ over the field $\Q(\zeta_9)$ induces an action over the set of solutions, therefore on the set of representations. Recall that $\Gal(\Q(\zeta_9)/\Q)$ is a cyclic group, and we will denote its generator by $g$. In particular, $g(\zeta_9)=\zeta^2_9$.
Therefore the element $g^2$ sends $\zeta_9$ to $\zeta_9^4,$ and one can verify that \[g^2(\omega)=g^2(\zeta_9^3)=\zeta_9^3,\] and so the group generated by $g^2$ is the unique subgroup of the Galois group that fixes $\omega$. Therefore, the subgroup $\langle g^2\rangle$ fixes the subfield $\Q(\omega)$ and in particular, all solutions $\{\alpha_ j\}_{j=1}^6$ are fixed points under the Galois group action.

By a direct computation, one can check that the action of $\langle g^2\rangle$ on the solutions $\{\alpha_ j\}_{j=7}^{15}$ has three closed orbits, namely $\{\alpha_7,\alpha_8,\alpha_9\},$ $\{\alpha_{10},\alpha_{11},\alpha_{12}\}$ and $\{\alpha_{13},\alpha_{14},\alpha_{15}\},$ and these are the only closed orbits in the set of solutions under the Galois group action.

\subsection{The Hermitian form}

In what follows, we will find the Hermitian form that is preserved by each representation and it will sometimes be degenerate. Assume that $H$ is an Hermitian form. 
Now if $\Gamma$ preserve $H,$ then the generators in the presentation \eqref{eq:RPres(3,6)} satisfy

\begin{equation*}
J^* H J = H,\ \ \ \
R_1^* H R_1 = H.
\end{equation*}

From the first equation, we can assure that $H$ is of the form
\begin{equation*}
\begin{bmatrix}
a & -\overline{c} & c \\
-c & a & \overline{c} \\
\overline{c} & c & a
\end{bmatrix},
\end{equation*}
where $a\in\R,$ and we can assume that $c\in \Q(\zeta_9)$. Let $c=b_0 +\Sigma_{j=1}^5 b_j \zeta_9^j$. Then the second equation provides a polynomial system on $\Q(\zeta_9)[a,b_0,b_1,b_2,b_3,b_4,b_5]$. For each solution $\alpha_i$, we use Gr\"obner basis to reduce the system into two linear equations.

%{\color{red} Is it important to provide the values of $a$ and $c$ in each case?}

Once we solved the equation, we have that for the Galois orbits $\{\alpha_7,\alpha_8,\alpha_9\}$ and $\{\alpha_{10},\alpha_{11},\alpha_{12}\}$, the Hermitian form is non-degenerate and of signature (3,0). For the orbit $\{\alpha_{13},\alpha_{14},\alpha_{15}\}$ the Hermitian form is non-degenerate of signature (2,1). In the case of the Galois fixed points $\{\alpha_1,\alpha_2,\alpha_3, \alpha_4, \alpha_5, \alpha_6\},$ all Hermitian forms are degenerate with two zero eigenvalues.

\subsection{The cusp group}

For a non-compact lattice, one defines a cusp holonomy as a conjugacy class of maximal subgroups (not containing loxodromic elements) fixing a point in the boundary of complex hyperbolic space. The following proposition was proved using a fundamental domain for the lattice (see \cite{boadiparker} and \cite{Irene}).

\begin{prop}
The Deligne-Mostow (3,6)-lattice $\Gamma_{(3,6)}\subset \PU(2,1)$ has only one cusp. The cusp holonomy is the class determined by $\Gamma_{cusp}= \langle R_2=JR_1 J^2 ,\, A_1=JR_1^2J^2R_1^2J\rangle$.
\end{prop}

The cusp holonomy might contain elliptic elements. The purely parabolic cusp holonomy is the maximal subgroup of the cusp holonomy with no elliptics.

\begin{prop}
The purely parabolic cusp holonomy of $\Gamma_{(3,6)}$ is $\langle [A_1,R_2],[A_1,R_2^2],(R_2 A_1)^2\rangle$. The centre of this group is the cyclic group generated by $ (R_2 A_1)^2$. 
\end{prop}

A straightforward computation with the list of representations gives the following description of cusp groups.
The representations for which the generator of the centraliser is elliptic, may be factorised (up to a finite index subgroup) through a representation of the fundamental group of the Satake-Baily-Borel compactification. Indeed, in the compactification, the centraliser of the cusp holonomy disappears.

\begin{prop}
The generator of the centraliser, $ (R_2 A_1)^2$, is elliptic of order at most three for all representations except for
three (up to conjugation) where it can be chosen to be unipotent of the form
\begin{align}
\rho( (R_2 A_1)^2)&= 
\begin{bmatrix}
1 & \xi & 0 \\
0 & 1& 0\\
0 & 0 & 1
\end{bmatrix}.
\end{align} 
 
\end{prop}

\addcontentsline{toc}{section}{\refname}
\bibliographystyle{alpha}
\bibliography{biblio}
%\nocite{*}

\begin{flushleft}
  \textsc{E. Falbel\\
  Institut de Math\'ematiques \\
  de Jussieu-Paris Rive Gauche \\
CNRS UMR 7586 and INRIA EPI-OURAGAN \\
 Sorbonne Universit\'e, Facult\'e des Sciences \\
4, place Jussieu 75252 Paris Cedex 05, France \\}
 \verb|elisha.falbel@imj-prg.fr|
 \end{flushleft}
 
  \begin{flushleft}
  \textsc{I. Pasquinelli\\
  University of Bristol  \\
School of Mathematics\\
Fry Building 
Woodland Road 
Bristol BS8 1UG 
UK
 \\}
 \verb|irene.pasquinelli@bristol.ac.uk|
 \end{flushleft}
 
 \begin{flushleft}
  \textsc{A. Ucan-Puc\\
  Institut de Math\'ematiques \\
  de Jussieu-Paris Rive Gauche \\
CNRS UMR 7586 \\
 Sorbonne Universit\'e, Facult\'e des Sciences \\
4, place Jussieu 75252 Paris Cedex 05, France \\}
 \verb|alejandro.ucan-puc@imj-prg.fr|
 \end{flushleft}

\end{document}